\documentclass{amsart}
\usepackage{mathptmx, bm, amscd, amssymb, enumerate, colonequals, color}
\definecolor{chianti}{rgb}{0.6,0,0}
\definecolor{meretale}{rgb}{0,0,.6}
\definecolor{leaf}{rgb}{0,.35,0}
\usepackage[colorlinks=true, pagebackref, hyperindex, citecolor=meretale, urlcolor=leaf, linkcolor=chianti]{hyperref}

\newtheorem{theorem}{Theorem}[section]
\newtheorem{corollary}[theorem]{Corollary}
\numberwithin{equation}{theorem}

\def\Spec{\operatorname{Spec}}

\def\frakm{\mathfrak{m}}
\def\frakn{\mathfrak{n}}
\def\frakp{\mathfrak{p}}
\def\frakM{\mathfrak{M}}
\def\frakN{\mathfrak{N}}

\def\FFp{\mathbb{F}_{\!p}}
\def\NN{\mathbb{N}}
\def\ZZ{\mathbb{Z}}

\def\calP{\mathcal{P}}

\def\balpha{{\boldsymbol{\alpha}}}
\def\bgamma{{\boldsymbol{\gamma}}}

\def\ge{\geqslant}
\def\le{\leqslant}
\def\bar{\overline}
\def\to{\longrightarrow}
\def\mapsto{\longmapsto}

\begin{document}
\title[Flat morphisms with regular fibers do not preserve $F$-rationality]{Flat morphisms with regular fibers\\ do not preserve $F$-rationality}

\author{Eamon Quinlan-Gallego}
\address{Department of Mathematics, University of Utah, Salt Lake City, UT~84112, USA}
\email{eamon.quinlan@utah.edu}

\author{Austyn Simpson}
\address{Department of Mathematics, University of Michigan, Ann Arbor, MI 48109 USA}
\email{austyn@umich.edu}

\author{Anurag K. Singh}
\address{Department of Mathematics, University of Utah, Salt Lake City, UT~84112, USA}
\email{anurag.singh@utah.edu}

\thanks{Quinlan-Gallego was supported by the NSF RTG grant DMS~\#1840190 and NSF postdoctoral fellowship DMS~\#2203065, Simpson by NSF postdoctoral fellowship DMS~\#2202890, and Singh by NSF grants DMS~\#1801285, DMS~\#2101671, and DMS~\#2349623.}

\begin{abstract}
For each prime integer $p>0$ we construct a standard graded $F$-rational ring~$R$, over a field $K$ of characteristic $p$, such that $R\otimes_K\overline{K}$ is not $F$-rational. By localizing we obtain a flat local homomorphism $(R, \mathfrak{m}) \to (S, \mathfrak{n})$ such that $R$ is $F$-rational,~$S/\mathfrak{m} S$ is regular (in fact, a field), but $S$ is not $F$-rational. In the process we also obtain standard graded $F$-rational rings $R$ for which $R\otimes_K R$ is not $F$-rational.
\end{abstract}
\maketitle

%%%%%%%%%%%%%%%%%%%%%%%%%%%%%%%%%%%%%%%%%%%%%%%%%%%%%%%%%%%%%%%%%%%%%%%%
\section{Introduction}
\label{section:intro}
%%%%%%%%%%%%%%%%%%%%%%%%%%%%%%%%%%%%%%%%%%%%%%%%%%%%%%%%%%%%%%%%%%%%%%%%

Let $\calP$ denote a local property of noetherian rings. The following types of \emph{ascent} have been studied extensively; recall that for $K$ a field, a noetherian $K$-algebra~$A$ is \emph{geometrically regular} over $K$ if $A\otimes_K L$ is regular for each finite extension field $L$ of $K$.

\begin{enumerate}
\item[{$(\hyperref[ASC1]{\textup{ASC}_{\textup{I}}})$}]
\label{ASC1}
For a flat local homomorphism $(R,\frakm)\to (S,\frakn)$ of excellent local rings, if $R$ is $\calP$ and the closed fiber $S/\frakm S$ is regular, then $S$ is $\calP$.

\item[{$(\hyperref[ASC2]{\textup{ASC}_{\textup{II}}})$}]
\label{ASC2}
For a flat local homomorphism $(R,\frakm)\to (S,\frakn)$ of excellent local rings, if $R$ is $\calP$ and the closed fiber $S/\frakm S$ is geometrically regular over $R/\frakm$, then $S$ is $\calP$.
\end{enumerate}

Our main interest here is when $\calP$ is $F$-rationality, a property rooted in Hochster and Huneke's theory of tight closure \cite{HH:JAMS}: a local ring $(R,\frakm)$ of positive prime characteristic is \emph{$F$-rational} if $R$ is Cohen-Macaulay and each ideal generated by a system of parameters for $R$ is tightly closed. Smith~\cite{Smith:ratsing} proved that $F$-rational rings have rational singularities, while Hara~\cite{Hara:AJM} and Mehta-Srinivas~\cite{Mehta:Srinivas} independently proved that rings with rational singularities have $F$-rational type. Rational singularities of characteristic zero satisfy~$(\hyperref[ASC1]{\textup{ASC}_{\textup{I}}})$, as proven by Elkik~\cite[Th\'eor\`eme~5]{Elkik}.

In the situation of $(\hyperref[ASC2]{\textup{ASC}_{\textup{II}}})$, geometric regularity of the closed fiber $R/\frakm\to S/\frakm S$ implies that of each fiber $k(\frakp)\to S\otimes_R k(\frakp)$ for $\frakp\in\Spec R$, see \cite[Page~297]{Andre}. The ascent~$(\hyperref[ASC2]{\textup{ASC}_{\textup{II}}})$ holds for $F$-rationality; this, and its variations, are due to V\'elez \cite[Theorem~3.1]{Velez}, Enescu~\cite[Theorem~2.27]{Enescu:2000}, Hashimoto~\cite[Theorem~6.4]{Hashimoto:cmfi}, and Aberbach-Enescu~\cite[Theorem~4.3]{Aberbach:Enescu}. A common thread amongst these is that each affirmative answer requires assumptions along the lines that the fibers are \emph{geometrically} regular. The situation is similar for $F$-injectivity in this regard; a local ring $(R,\frakm)$ of positive prime characteristic is \emph{$F$-injective} if the Frobenius action on local cohomology modules
\[
F\colon H^k_\frakm(R)\to H^k_\frakm(R)
\]
is injective for each $k\ge 0$. Datta and Murayama~\cite[Theorem~A]{Dutta:Murayama} proved that if $(R,\frakm)$ is~$F$-injective, and $(R,\frakm)\to (S,\frakn)$ is a flat local map such that $S/\frakm S$ is Cohen-Macaulay and \emph{geometrically} $F$-injective over $R/\frakm$, then $S$ is $F$-injective; see also~\cite[Theorem~4.3]{Enescu:2009} and \cite[Corollary~5.7]{Hashimoto:cmfi}. We present examples demonstrating that the geometric assumptions are indeed required, i.e., that $F$-rationality and~$F$-injectivity do not satisfy $(\hyperref[ASC1]{\textup{ASC}_{\textup{I}}})$:

\begin{theorem}
\label{theorem:main:theorem}
For each prime integer $p>0$, there exists a flat local ring homomorphism~$(R,\frakm)\to (S,\frakn)$ of excellent local rings of characteristic $p$, such that the ring $R$ is~$F$-rational, $S/\frakm S$ is regular, but $S$ is not $F$-rational or even $F$-injective.
\end{theorem}

Enescu had earlier demonstrated that $F$-injectivity does not satisfy $(\hyperref[ASC1]{\textup{ASC}_{\textup{I}}})$, though the examples~\cite[Page~3075]{Enescu:2009} are not normal; the question of whether normal $F$-injective rings satisfy $(\hyperref[ASC1]{\textup{ASC}_{\textup{I}}})$ has been raised earlier, e.g.,~\cite[Question 8.1]{Schwede:Zhang}, and is settled in the negative by Theorem~\ref{theorem:main:theorem}. There is a more recent notion, \emph{$F$-anti-nilpotence}, developed in the papers~\cite{Enescu:Hochster, Ma, Ma:Quy}; in view of the implications
\[
\text{$F$-rational \ $\implies$ \ $F$-anti-nilpotent \ $\implies$ $F$-injective},
\]
Theorem~\ref{theorem:main:theorem} also shows that $F$-anti-nilpotence does not satisfy $(\hyperref[ASC1]{\textup{ASC}_{\textup{I}}})$.

It is worth mentioning that the rings $R$ in Theorem~\ref{theorem:main:theorem} are necessarily not Gorenstein, since $F$-rational Gorenstein rings are $F$-regular by~\cite[Theorem~4.2]{HH:TAMS}, and $F$-regularity satisfies $(\hyperref[ASC1]{\textup{ASC}_{\textup{I}}})$ by~\cite[Theorem~3.6]{Aberbach}. Another subtlety is that such examples can only exist over imperfect fields, since $(\hyperref[ASC1]{\textup{ASC}_{\textup{I}}})$ and $(\hyperref[ASC2]{\textup{ASC}_{\textup{II}}})$ coincide when $R/\frakm$ is a perfect field, and $F$-rationality satisfies $(\hyperref[ASC2]{\textup{ASC}_{\textup{II}}})$.

Preliminary results are recorded in \S\ref{section:prelim}, including an extension of a criterion for $F$-rationality due to Fedder and Watanabe~\cite{Fedder:Watanabe}. In \S\ref{section:main} we construct two families of examples that each imply Theorem~\ref{theorem:main:theorem}: the first has the advantage that the proofs are more transparent, though the transcendence degree of the imperfect field over $\FFp$ increases with the characteristic~$p$; the second family accomplishes the desired with transcendence degree one, independent of the characteristic $p>0$, though the calculations are more involved. The examples in~\S\ref{section:main} are constructed as standard graded rings, with the relevant properties preserved under passing to localizations. In the process, we also obtain standard graded $F$-rational rings $R$, with the degree zero component being a field $K$ of positive characteristic, such that the enveloping algebra $R\otimes_K R$ is not $F$-rational.

%%%%%%%%%%%%%%%%%%%%%%%%%%%%%%%%%%%%%%%%%%%%%%%%%%%%%%%%%%%%%%%%%%%%%%%%
\section{Preliminaries}
\label{section:prelim}
%%%%%%%%%%%%%%%%%%%%%%%%%%%%%%%%%%%%%%%%%%%%%%%%%%%%%%%%%%%%%%%%%%%%%%%%

Following \cite[page~125]{Hochster:tc}, a local ring of positive prime characteristic is \emph{$F$-rational} if it is a homomorphic image of a Cohen-Macaulay ring, and each ideal generated by a system of parameters is tightly closed. It follows from this definition that an $F$-rational local ring is Cohen-Macaulay,~\cite[Theorem~4.2]{HH:TAMS}, so the notion coincides with that in~\S\ref{section:intro}. Moreover, an $F$-rational local ring is a normal domain. A localization of an $F$-rational local ring at a prime ideal is again $F$-rational; with this in mind, a noetherian ring of positive prime characteristic---which is not necessarily local---is \emph{$F$-rational} if its localization at each maximal ideal (equivalently, at each prime ideal) is~$F$-rational.

For the case of interest in this paper, let $R$ be an $\NN$-graded Cohen-Macaulay normal domain, such that the degree zero component is a field $K$ of characteristic $p>0$, and~$R$ is a finitely generated $K$-algebra. Then~$R$ is $F$-rational if and only if the ideal generated by some (equivalently, any) homogeneous system of parameters for~$R$ is tightly closed; see~\cite[Theorem~4.7]{HH:JAG} and the remark preceding it. An equivalent formulation in terms of local cohomology, following \cite[Proposition~3.3]{Smith:invent}, is described next:

Fix a homogeneous system of parameters $x_1,\dots,x_d$ for $R$, i.e., a sequence of~$d\colonequals\dim R$ homogeneous elements that generate an ideal with radical the homogeneous maximal ideal~$\frakm$ of $R$. The local cohomology module $H^d_\frakm(R)$ may then be computed using a \v Cech complex on $x_1,\dots,x_d$ as
\[
H^d_\frakm(R)\ =\ \frac{R_{x_1\cdots x_d}}{\sum\limits_i R_{x_1\cdots \hat{x_i}\cdots x_d}}.
\]
This module admits a natural $\ZZ$-grading where the cohomology class
\begin{equation}
\label{equation:prelim:eta}
\eta\colonequals\left[\frac{r}{x_1^k\cdots\,x_d^k}\right]\ \in\ H^d_\frakm(R),
\end{equation}
for $r\in R$ a homogeneous element, has
\[
\deg\eta \colonequals \deg r-k\sum_{i=1}^d\deg x_i.
\]
The Frobenius endomorphism $F\colon R\to R$ induces a map
\[
F\colon H^d_\frakm(R)\ \to\ H^d_{F(\frakm)}(R)=H^d_\frakm(R)
\]
that is the \emph{Frobenius action} on $H^d_\frakm(R)$; this is simply the map
\begin{equation}
\label{equation:prelim:action}
\eta=\left[\frac{r}{x_1^k\cdots\,x_d^k}\right]\ \mapsto\ F(\eta)=\left[\frac{r^p}{x_1^{kp}\cdots\,x_d^{kp}}\right].
\end{equation}
Since $R$ is Cohen-Macaulay by assumption, $R$ is $F$-injective precisely when the map~\eqref{equation:prelim:action} is injective.

The element $\eta$ as in~\eqref{equation:prelim:eta} belongs to $0^*_{H^d_\frakm(R)}$, the \emph{tight closure} of zero in $H^d_\frakm(R)$, if there exists a nonzero element $c\in R$ such that for all $e\in\NN$, one has
\[
cF^e(\eta)\ =\ 0
\]
in $H^d_\frakm(R)$. This translates as
\[
cr^{p^e}\ \in\ (x_1^{kp^e},\dots,\ x_d^{kp^e})R
\]
for all $e\in\NN$. In particular, $R$ is $F$-rational precisely when
\[
0^*_{H^d_\frakm(R)}\ =\ 0.
\]
It follows that an $F$-rational ring must be $F$-injective.

We next review Veronese subrings. Let $S$ be an $\NN$-graded ring for which the degree zero component is a field $K$, and~$S$ is a finitely generated $K$-algebra. Fix a positive integer $n$. Then the $n$-th \emph{Veronese subring} of $S$ is the ring
\[
S^{(n)}\colonequals\bigoplus_{k\in\NN}S_{nk}.
\]
Set $R\colonequals S^{(n)}$. The extension $R\subseteq S$ is split, so if $S$ is normal ring then so is $R$. Let $\frakm$ denote the homogeneous maximal ideal of~$R$, and note that $\frakm S$ is primary to the homogeneous maximal ideal~$\frakn$ of $S$. For all $i\le d \colonequals \dim S = \dim R$, it follows that $H^i_\frakm(R)$ is a direct summand of $H^i_{\frakm}(S)=H^i_{\frakn}(S)$, and hence that the ring $R$ is Cohen-Macaulay whenever $S$ is. Moreover, by ~\cite[Theorem~3.1.1]{Goto:Watanabe} one has
\[
H^d_\frakm(R)\ =\ \bigoplus_{k\in\ZZ}{[H^d_\frakn(S)]}_{nk}.
\]

Suppose $S \colonequals K[x_0,\dots, x_d]/(f)$, where $f$ is a homogeneous polynomial that is monic of degree $m$ with respect to the indeterminate $x_0$. Then $S$ is free over the polynomial subring $K[x_1,\dots,x_d]$, with basis $\{1,x_0,\dots,x_0^{m-1}\}$. The local cohomology module~$H^d_\frakn(S)$, as computed using a \v Cech complex on $x_1,\dots,x_d$, thus has a $K$-basis consisting of elements
\begin{equation}
\label{equation:basis}
\left[\frac{x_0^{\alpha_0}}{x_1^{\alpha_1+1}\ \cdots\ \,x_d^{\alpha_d+1}}\right]\ \in\ H^d_\frakn(S)
\end{equation}
where each $\alpha_i$ is a nonnegative integer, and $\alpha_0\le m-1$. When $S$ is graded, by restricting to elements of appropriate degree, one obtains a basis for a graded component of $H^d_\frakn(S)$, or for the local cohomology~$H^d_\frakm(R)$ of the Veronese subring $R$. Similarly, for the enveloping algebra $S\otimes_K S$, one has a $K$-basis as follows: use $y_0,\dots,y_d$ for the second copy of $S$, and consider the maximal ideal $\frakN\colonequals(x_0,\dots,x_d,\,y_0,\dots,y_d)$ of $S\otimes_K S$. Then the local cohomology module~$H^{2d}_\frakN(S\otimes_K S)$ has a $K$-basis
\begin{equation}
\label{equation:basis:enveloping}
\left[\frac{x_0^{\alpha_0}y_0^{\beta_0}}{x_1^{\alpha_1+1}\ \cdots\ \,x_d^{\alpha_d+1}y_1^{\beta_1+1}\ \cdots\ \,y_d^{\beta_d+1}}\right]
\end{equation}
where each $\alpha_i,\beta_j$ is a nonnegative integer, $\alpha_0\le m-1$, and $\beta_0\le m-1$.

The following is a variation of~\cite[Theorem~2.8]{Fedder:Watanabe} and~\cite[Theorem~7.12]{HH:JAG}, and is used in the proof of Theorem~\ref{theorem:small:transc}.

\begin{theorem}
\label{theorem:fedder:watanabe}
Let $S$ be an $\NN$-graded Cohen-Macaulay normal domain, such that the degree zero component is a field $K$ of positive characteristic, and~$S$ is a finitely generated~$K$-algebra. Let $\frakn$ denote the homogeneous maximal ideal of~$S$, and set $d\colonequals\dim S$.

Suppose each nonzero element of $\frakn$ has a power that is a test element, and that there exists an integer $n>0$ such that the Frobenius action on
\[
{[H^d_\frakn(S)]}_{\le -n}
\]
is injective. Then the tight closure of zero in $H^d_\frakn(S)$ is contained in ${[H^d_\frakn(S)]}_{>-n}$.
\end{theorem}

\begin{proof}
The hypotheses ensure that $S$ has a homogeneous system of parameters $x_1,\dots,x_d$ where each $x_i$ is a test element; we compute local cohomology using a \v Cech complex on such a homogeneous system of parameters. Suppose the assertion of the theorem is false; then there exists a nonzero homogeneous element $\eta$ in $0^*_{H^d_\frakn(S)}$ with~$\deg\eta\le -n$. After possibly replacing the $x_i$ by powers, we may assume that
\[
\eta\ =\ \left[\frac{s}{x_1\cdots\,x_d}\right],
\]
for $s$ a homogeneous element of $S$. Since each $x_i$ is a test element, one has
\[
x_is^q\ \in\ (x_1^q,\dots,x_d^q)
\]
for each $q=p^e$, and hence
\[
s^q\ \in\ (x_1^q,\dots,x_d^q):_R (x_1,\dots,x_d)\ =\ (x_1^q,\dots,x_d^q) + (x_1\cdots\,x_d)^{q-1},
\]
where the equality is because $x_1,\dots,x_d$ is a regular sequence. Since $F^e(\eta)$ is nonzero in view of the injectivity of the Frobenius action on ${[H^d_\frakn(S)]}_{\le -n}$, one has
\[
s^q\ \notin\ (x_1^q,\dots,x_d^q).
\]
This implies that $\deg s^q\ge \deg (x_1\cdots\,x_d)^{q-1}$ for each $q=p^e$, which translates as
\[
\deg s\ \ge\ \frac{q-1}{q}\deg (x_1\cdots\,x_d).
\]
Taking the limit $e\to\infty$ gives
\[
\deg s\ \ge\ \deg (x_1\cdots\,x_d),
\]
so $\deg\eta\ge 0$. This contradicts $\deg\eta\le -n<0$.
\end{proof}

A ring $S$ is \emph{standard graded} if it is $\NN$-graded, with the degree zero component being a field $K$, such that $S$ is generated as a $K$-algebra by finitely many elements of $S_1$.

While Theorem~\ref{theorem:fedder:watanabe} requires the injectivity of the Frobenius action on ${[H^d_\frakn(S)]}_{\le -n}$, additional hypotheses enable one to verify the injectivity of Frobenius on \emph{one} graded component; the following corollary will be used in the proof of Theorem~\ref{theorem:small:transc}. Following~\cite{Goto:Watanabe}, the \emph{$a$-invariant} of a Cohen-Macaulay graded ring $S$, as in Theorem~\ref{theorem:fedder:watanabe}, is
\[
a(S)\colonequals\max\left\{i\in\ZZ\ \big{|}\ {[H^d_\frakn(S)]}_i\neq 0\right\}.
\]

\begin{corollary}
\label{corollary:fedder:watanabe}
Let $S$ be a standard graded Gorenstein normal domain, of characteristic~$p>0$, such that the homogeneous maximal ideal $\frakn$ is an isolated singular point. Set~$d\colonequals\dim S$. Suppose $a(S)<0$, and that there exists an integer $n$ with $-n\le a(S)$ such that the Frobenius action
\[
F\colon {[H^d_\frakn(S)]}_{-n}\to {[H^d_\frakn(S)]}_{-np}
\]
is injective. Then the Veronese subring $S^{(n)}$ is $F$-rational.
\end{corollary}

\begin{proof}
Because $\frakn$ is an isolated singular point, each nonzero element of $\frakn$ has a power that is a test element, and Theorem~\ref{theorem:fedder:watanabe} is applicable. Since $S$ is Gorenstein, each nonzero homogeneous element $\eta$ of~${[H^d_\frakn(S)]}_{\le -n}$ has a nonzero multiple $s\eta$ in the socle of~$H^d_\frakn(S)$, which is the graded component~${[H^d_\frakn(S)]}_{a(S)}$. As~$S$ is standard graded, such a multiplier~$s\in S$ can be chosen to be a product of elements of degree one, therefore $\eta$ has a nonzero multiple~$s'\eta$ in ${[H^d_\frakn(S)]}_{-n}$. Since $F(s'\eta)$ is nonzero, so is $F(\eta)$. It follows that the Frobenius action on ${[H^d_\frakn(S)]}_{\le -n}$ is injective, so Theorem~\ref{theorem:fedder:watanabe} implies that the tight closure of zero in~$H^d_\frakn(S)$ is contained in~${[H^d_\frakn(S)]}_{>-n}$.

Set $R\colonequals S^{(n)}$. The hypotheses $-n\le a(S)<0$ give
\[
H^d_\frakm(R)\ \subseteq\ {[H^d_\frakn(S)]}_{\le -n}
\]
where $\frakm$ is the homogeneous maximal ideal of~$R$. As the tight closure of zero in~$H^d_\frakm(R)$ is contained in the tight closure of zero in $H^d_\frakn(S)$, the assertion follows.
\end{proof}

%%%%%%%%%%%%%%%%%%%%%%%%%%%%%%%%%%%%%%%%%%%%%%%%%%%%%%%%%%%%%%%%%%%%%%%%
\section{The examples}
\label{section:main}
%%%%%%%%%%%%%%%%%%%%%%%%%%%%%%%%%%%%%%%%%%%%%%%%%%%%%%%%%%%%%%%%%%%%%%%%

\begin{theorem}
\label{theorem:trans:p}
Fix a prime integer $p>0$. Let $t_1,\dots,t_p$ be indeterminates over the field $\FFp$ and set $K\colonequals\FFp(t_1,\dots,t_p)$. Consider the hypersurface
\[
S\colonequals K[x_0,\dots,x_p]/(x_0^p - t^{\phantom{p}}_1\!x_1^p - \dots - t^{\phantom{p}}_p x_p^p)
\]
with the standard $\NN$-grading, and its $p$-th Veronese subring $R\colonequals S^{(p)}$. Then:
\begin{enumerate}[\quad\rm(1)]
\item\label{trans:p:1} The ring $R$ is $F$-rational.
\item\label{trans:p:2} The rings $R\otimes_K K^{1/p}$ and $R\otimes_K\bar{K}$ are not $F$-injective, hence not $F$-rational.
\item\label{trans:p:3} The enveloping algebra $R\otimes_K R$ is not $F$-injective, hence not $F$-rational.
\end{enumerate}
\end{theorem}

\begin{proof}
First consider the hypersurface
\[
A\colonequals\FFp[t_1,\dots,t_p,\,x_0,\dots,x_p]/(x_0^p - t^{\phantom{p}}_1\!x_1^p - \dots - t^{\phantom{p}}_p x_p^p).
\]
The Jacobian criterion shows $A_{x_i}$ is regular for each $i$, so $A$ is normal by Serre's criterion. By inverting an appropriate multiplicative set in $A$, one obtains the ring $S$, which therefore is also normal. Since $R$ is a pure subring of the finite extension ring $S$, it follows that $R$ is normal and Cohen-Macaulay.

Note that $S$ is not $F$-injective: set $\frakn$ to be the homogeneous maximal ideal of $S$; computing local cohomology $H^p_\frakn(S)$ using a \v Cech complex on the system of parameters $x_1,\dots,x_p$ for $S$, the cohomology class
\[
\left[\frac{x_0}{x_1\cdots\,x_p}\right]\ \in\ H^p_\frakn(S)
\]
maps to zero under the Frobenius action on $H^p_\frakn(S)$. We shall see that the Frobenius action on $H^p_\frakm(R)$, with $\frakm$ the homogeneous maximal ideal of $R$, is however injective.

First note that ${[H^p_\frakm(R)]}_{-p}$ is the socle of $H^p_\frakm(R)$: it is the highest degree component, and any nonzero homogeneous element $\eta\in H^p_\frakm(R)$ has a nonzero multiple $s\eta$ in the socle of~$H^p_\frakn(S)$, which is ${[H^p_\frakn(S)]}_{-1}$; but then it has a nonzero multiple $s'\eta$ in
\[
{[H^p_\frakn(S)]}_{-p}={[H^p_\frakm(R)]}_{-p},
\]
for $s,s'$ homogeneous in $S$, in which case degree considerations imply that $s'\in R$.

To verify that the Frobenius action $F$ on $H^p_\frakm(R)$ is injective, it suffices to prove the injectivity of $F$ on the socle ${[H^p_\frakm(R)]}_{-p}$ which, following~\eqref{equation:basis}, is the~$K$-vector space spanned by the cohomology classes
\[
\eta_{\balpha}\colonequals\left[\frac{x_0^{\alpha_1+\dots+\alpha_p}}{x_1^{\alpha_1+1}\ \cdots\ \,x_p^{\alpha_p+1}}\right]\ \in\ {[H^p_\frakm(R)]}_{-p},
\]
where each $\alpha_i$ is a nonnegative integer, $\sum\alpha_i\le p-1$, and $\balpha\colonequals(\alpha_1,\dots,\alpha_p)$. Since
\[
x_0^p\ =\ t^{\phantom{p}}_1\!x_1^p + \dots + t^{\phantom{p}}_p x_p^p
\]
in the ring $S$, one has
\begin{equation}
\label{equation:F:eta}
F(\eta_{\balpha})
\ =\ 
\left[\frac{{(t^{\phantom{p}}_1\!x_1^p + \dots + t^{\phantom{p}}_p x_p^p)}^{\sum\alpha_i}}{x_1^{p\alpha_1+p}\ \cdots\ \,x_p^{p\alpha_p+p}}\right]
\ =\
\frac{(\sum\alpha_i)!}{\alpha_1!\cdots\alpha_p!}
\left[\frac{t_1^{\alpha_1}\cdots\ t_p^{\alpha_p}}{x_1^p\ \cdots\ \,x_p^p}\right],
\end{equation}
where the latter equality uses the pigeonhole principle. The elements $t_1^{\alpha_1}\cdots\ t_p^{\alpha_p}$ of~$K$, as~$\balpha$ varies subject to the conditions above, are linearly independent over the subfield~$K^p$. It follows that for any nonzero $K$-linear combination $\eta$ of the elements $\eta_{\balpha}$, one has $F(\eta)\neq 0$. This proves that the ring $R$ is $F$-injective.

One may now use Corollary~\ref{corollary:fedder:watanabe} to conclude that $R$ is $F$-rational; alternatively, one can also argue as follows: Equation~\eqref{equation:F:eta} shows that the image of ${[H^p_\frakm(R)]}_{-p}$ under $F$ lies in the $K$-span of the cohomology class
\[
\mu\colonequals\left[\frac{1}{x_1^p\cdots\,x_p^p}\right], 
\]
so it suffices to verify that $\mu$ does not belong to the tight closure of zero in $H^p_\frakm(R)$. This holds since no nonzero homogeneous form in $R$ annihilates
\[
F^e(\mu)\ =\ \left[\frac{1}{x_1^{p^{e+1}}\cdots\ \,x_p^{p^{e+1}}}\right]
\]
for each $e\ge 0$.

For~\eqref{trans:p:2}, let $\bar{R}$ denote either of $R\otimes_K K^{1/p}$ or $R\otimes_K\bar{K}$. Note that
\[
t_2^{1/p}\left[\frac{x_0}{x_1^2x_2\cdots x_p}\right]-t_1^{1/p}\left[\frac{x_0}{x_1x_2^2x_3\cdots x_p}\right]
\]
is a nonzero element of $H^p_\frakm(\bar{R})$, since it is a nontrivial linear combination of basis elements as in~\eqref{equation:basis}. However its image under the Frobenius action is
\[
t_2\left[\frac{t^{\phantom{p}}_1\!x_1^p + \dots + t^{\phantom{p}}_p x_p^p}{x_1^{2p}x_2^p\cdots x_p^p}\right] - 
t_1\left[\frac{t^{\phantom{p}}_1\!x_1^p + \dots + t^{\phantom{p}}_p x_p^p}{x_1^px_2^{2p}x_3^p\cdots x_p^p}\right]
\ = \ 
t_2\left[\frac{t_1}{x_1^px_2^p\cdots x_p^p}\right] - t_1\left[\frac{t_2}{x_1^px_2^p\cdots x_p^p}\right]
\]
which, of course, is zero.

Lastly, for~\eqref{trans:p:3}, write the enveloping algebra $S\otimes_K S$ of $S$ as
\[
K[x_0,\dots,x_p,\,y_0,\dots,y_p]/(x_0^p - t^{\phantom{p}}_1\!x_1^p - \dots - t^{\phantom{p}}_p x_p^p,\ 
y_0^p - t^{\phantom{p}}_1\!y_1^p - \dots - t^{\phantom{p}}_p y_p^p),
\]
with the $\NN^2$-grading under which $\deg x_i=(1,0)$ and $\deg y_i=(0,1)$ for each $i$. Then
\[
R\otimes_K R\ =\ \bigoplus_{k,l\in\NN}{[S\otimes_K S]}_{(pk,pl)}.
\]
Note that $R\otimes_K R$ admits a standard grading; let $\frakM$ denote its homogeneous maximal ideal. Then $\frakM(S\otimes_K S)$ is primary to the homogeneous maximal ideal $\frakN\colonequals(x_0,\dots,x_p,\,y_0,\dots,y_p)$ of $S\otimes_K S$, and 
\[
H^{2p}_\frakM(R\otimes_K R)\ =\ \bigoplus_{k,l\in\NN}{[H^{2p}_\frakN(S\otimes_K S)]}_{(pk,pl)}.
\]
The cohomology class
\[
\left[\frac{x_0y_1-x_1y_0}{x_1^2x_2\cdots x_p\,y_1^2y_2\cdots y_p}\right]\ \in\ H^{2p}_\frakM(R\otimes_K R)
\]
is nonzero since it is a nontrivial linear combination of basis elements as in~\eqref{equation:basis:enveloping}; however, it is readily seen to be in the kernel of the Frobenius action.
\end{proof}

Note that $R\otimes_K K^{1/p}$ and $R\otimes_K\bar{K}$ in the previous theorem are not reduced: for example,
\[
(x_0 - t_1^{1/p}x_1 - \dots - t_p^{1/p} x_p)\, x_1\cdots\, x_{p-1}
\]
is a nonzero nilpotent element. This gives an alternative proof of~\eqref{trans:p:2}, since $F$-injective rings are reduced by~\cite[Remark~2.6]{Schwede:Zhang}.

In the examples provided by Theorem~\ref{theorem:trans:p}, the transcendence degree of~$K$ over~$\FFp$ increases with $p$; for the interested reader, the following theorem gets around this, though the proof is perhaps more technical:

\begin{theorem}
\label{theorem:small:transc}
Fix a prime integer $p>0$. Let $t$ be an indeterminate over the field $\FFp$ and set $K\colonequals\FFp(t)$. Consider the hypersurface
\[
S\colonequals K[w,x,y,\,z_1,\dots,z_{p-1}]/(w^{p+1}-tx^{p+1}-xy^p-\sum_{i=1}^{p-1}z_i^{p+1})
\]
with the standard $\NN$-grading, and set $R\colonequals S^{(p)}$. Then:
\begin{enumerate}[\quad\rm(1)]
\item\label{small:transc:1} The ring $R$ is $F$-rational.
\item\label{small:transc:2} The rings $R \otimes_K K^{1/p}$ and $R \otimes_K \overline K$ are not $F$-injective, hence not $F$-rational.
\item\label{small:transc:3} The enveloping algebra $R\otimes_K R$ is not $F$-injective, hence not $F$-rational.
\end{enumerate}
\end{theorem}

\begin{proof}
We begin with the hypersurface
\[
A\colonequals\FFp[t,w,x,y,\,z_1,\dots,z_{p-1}]/(w^{p+1}-tx^{p+1}-xy^p-\sum_i z_i^{p+1}).
\]
The Jacobian criterion shows that, up to radical, the defining ideal of the singular locus of~$A$ contains~$(w,x,y,z_1,\dots,z_{p-1})$. The ring $S$ is obtained from $A$ by inverting an appropriate multiplicative set; it follows that $S$ has an isolated singular point at its homogeneous maximal ideal $\frakn$. In particular, $S$ is normal by Serre's criterion.

To prove that $R$ is $F$-rational, it suffices by Corollary~\ref{corollary:fedder:watanabe} to verify that
\begin{equation}
\label{equation:small:transc:action}
F\colon {[H^{p+1}_\frakn(S)]}_{-p} \to {[H^{p+1}_\frakn(S)]}_{-p^2}
\end{equation}
is injective. Using the~\v Cech complex on~$x,y,z_1\dots,z_{p-1}$, the vector space ${[H^{p+1}_\frakn(S)]}_{-p}$ has a $K$-basis, as in~\eqref{equation:basis}, consisting of cohomology classes
\[
\eta_{\alpha,\beta,\bgamma}\colonequals\left[\frac{w^{1+\alpha+\beta+\sum\gamma_i}}{x^{\alpha+1}\,y^{\beta+1}\,\prod\limits_i z_i^{\gamma_i+1}}\right]
\]
where $\alpha,\beta,\gamma_1,\dots,\gamma_{p-1}$ are nonnegative integers with $\alpha+\beta+\sum\gamma_i\ \le\ p-1$. The ring $S$ admits a $(\ZZ/(p+1))^{p+1}$-grading with
\[
\deg z_i = e_i,\ \ \deg w = e_p,\ \ \deg x = e_{p+1} = \deg y,
\]
where $e_1,\dots,e_{p+1}$ denote standard basis vectors modulo $p+1$. Since $\gcd(p,\ p+1)=1$, the action~\eqref{equation:small:transc:action} maps distinct multigraded components to distinct multigraded components, so it suffices to verify the injectivity componentwise. Note that 
\[
\deg\eta_{\alpha,\beta,\bgamma}\ =\ \Big(-\gamma_1-1,\ \dots,\ -\gamma_{p-1}-1,\ 1+\alpha+\beta+\sum\limits_i\gamma_i,\ -\alpha-\beta-2\Big)
\]
with respect to the multigrading. Thus, for fixed nonnegative integers $k$ and $\gamma_i$ with
\[
0\ \le\ k+\sum\limits_i\gamma_i\ \le\ p-1,
\]
a homogeneous element of ${[H^{p+1}_\frakn(S)]}_{-p}$ with multidegree
\[
\Big(-\gamma_1-1,\ \dots,\ -\gamma_{p-1}-1,\ 1+k+\sum\limits_i\gamma_i,\ -k-2\Big)
\]
has the form
\[
\sum_{\alpha+\beta = k} c_\alpha \eta_{\alpha,\beta,\bgamma},
\]
where $\alpha$ and $\beta$ are nonnegative integers with $\alpha+\beta=k$, and $c_\alpha\in K$.

Set $m\colonequals k+\sum\gamma_i$, and suppose that the above element
\begin{equation}
\label{equation:ker:element:1}
\sum_{\alpha+\beta = k} c_\alpha \eta_{\alpha,\beta,\bgamma}\ =\
\sum_{\alpha+\beta = k} c_\alpha x^\beta y^\alpha \left[\frac{w^{m+1}}{x^{k+1}\,y^{k+1}\,\prod\limits_i z_i^{\gamma_i+1}}\right]
\end{equation}
belongs to the kernel of the Frobenius action. Then
\[
\Big(\sum_{\alpha+\beta = k} c^p_\alpha x^{\beta\! p} y^{\alpha\! p}\Big) w^{(m+1)p}
\]
belongs to the ideal
\[
\Big(x^{(k+1)p},\ y^{(k+1)p},\ z_1^{(\gamma_1+1)p},\ \dots,\ z_{p-1}^{(\gamma_{p-1}+1)p}\Big)S.
\]
Since $w^{(m+1)p}\ =\ w^{p-m}\,w^{(p+1)m}$ and $1\le p-m \le p$, it follows that 
\begin{equation}
\label{equation:ker:element:2}
\Big(\sum_{\alpha+\beta = k} c^p_\alpha x^{\beta\! p} y^{\alpha\! p}\Big) \Big(tx^{p+1} + xy^p + \sum_{i=1}^{p-1}z_i^{p+1}\Big)^m
\end{equation}
belongs to the monomial ideal
\begin{equation}
\label{equation:ker:element:3}
\Big(x^{(k+1)p},\ y^{(k+1)p},\ z_1^{(\gamma_1+1)p},\ \dots,\ z_{p-1}^{(\gamma_{p-1}+1)p}\Big)
\end{equation}
in the polynomial ring~$K[x,y,\,z_1,\dots,z_{p-1}]$. Bearing in mind that $m=k+\sum\gamma_i$, the terms in the multinomial expansion of~\eqref{equation:ker:element:2} that include the monomial
\[
\prod\limits_i z_i^{(p+1)\gamma_i}
\]
constitute the polynomial 
\[
\binom{m}{k,\gamma_1,\dots,\gamma_{p-1}}\Big(\sum_{\alpha+\beta = k} c^p_\alpha x^{\beta\! p} y^{\alpha\! p}\Big) (tx^{p+1} + xy^p)^k \prod\limits_i z_i^{(p+1)\gamma_i}
\]
which, therefore, also belongs to the monomial ideal~\eqref{equation:ker:element:3}. But then
\[
\Big(\sum_{\alpha+\beta = k} c^p_\alpha x^{\beta\! p} y^{\alpha\! p}\Big) (tx^{p+1} + xy^p)^k
\ \in\ \big(x^{(k+1)p},\ y^{(k+1)p}\big)
\]
in the polynomial ring~$K[x,y]$. This implies that the coefficient of $x^{kp+k}y^{kp}$ in the polynomial above must be zero, i.e., that
\[
\sum_{\alpha+\beta = k}\binom{k}{\alpha} c_\alpha^p t^\alpha \ =\ 0.
\]
Since $c_\alpha^p\in K^p$ for each $\alpha$, and $k < [K^p(t):K^p]=p$, this forces each $c_\alpha$ to be zero. But then the element~\eqref{equation:ker:element:1} is zero, so the map~\eqref{equation:small:transc:action} is indeed injective as claimed. This completes the proof of~\eqref{small:transc:1}.

For~\eqref{small:transc:2}, let $\frakm$ denote the homogeneous maximal ideal of $R$, and let $\bar{R}$ denote either of~$R\otimes_K K^{1/p}$ or $R\otimes_K\bar{K}$. Then
\[
\left[\frac{w^2}{x^2y\prod\limits_i z_i}\right]\ -\ t^{1/p} \left[\frac{w^2}{xy^2\prod\limits_i z_i}\right]\ \in\ H^{p+1}_\frakm(\bar{R})
\]
is a nontrivial linear combination of basis elements as in~\eqref{equation:basis}. The ring $\bar{R}$ is not $F$-injective since under the Frobenius action on $H^{p+1}_\frakm(\bar{R})$, this element maps to
\[
\left[\frac{w^{p-1}tx}{x^py^p\prod\limits_i z_i^p}\right]\ -\ t\left[\frac{w^{p-1}x}{x^py^p\prod\limits_i z_i^p}\right]\ =\ 0.
\]

For~\eqref{small:transc:3} use $w',x',y',z_i'$ for the second copy of $S$, and proceed as in the proof of Theorem~\ref{theorem:trans:p}. Using $\frakM$ for the homogeneous maximal ideal of $R\otimes_K R$, the cohomology class
\[
\left[\frac{(ww')^2\,(x'y-xy')}{(xx'yy')^2\,\prod\limits_i z_i\,\prod\limits_i z_i'}\right]\ \in\ H^{2p+2}_\frakM(R\otimes_K R)
\]
is a nontrivial linear combination of basis elements as in~\eqref{equation:basis:enveloping}, and is in the kernel of the Frobenius action on $H^{2p+2}_\frakM(R\otimes_K R)$. It follows that the ring $R\otimes_K R$ is not $F$-injective.
\end{proof}

Theorem~\ref{theorem:main:theorem} follows readily from the results of this section:

\begin{proof}[Proof of Theorem~\ref{theorem:main:theorem}]
Let $K$ and $R$ be as in Theorem~\ref{theorem:trans:p} or~\ref{theorem:small:transc}, and let $S\colonequals R \otimes_K K^{1/p}$ or~$R \otimes_K \overline K$. An example is then obtained after localizing at the homogeneous maximal ideals; note that the closed fiber is the field $K^{1/p}$ or $\overline K$ in the respective cases.
\end{proof}

%%%%%%%%%%%%%%%%%%%%%%%%%%%%%%%%%%%%%%%%%%%%%%%%%%%%%%%%%%%%%%%%%%%%%%%%
\subsection*{Acknowledgments}
%%%%%%%%%%%%%%%%%%%%%%%%%%%%%%%%%%%%%%%%%%%%%%%%%%%%%%%%%%%%%%%%%%%%%%%%

We are grateful to Ofer Gabber, Linquan Ma, Karl Schwede, and Kei-ichi Watanabe for helpful conversations, and to the referees for a number of useful comments and suggestions.

%%%%%%%%%%%%%%%%%%%%%%%%%%%%%%%%%%%%%%%%%%%%%%%%%%%%%%%%%%%%%%%%%%%%%%%%

\end{document}